\documentclass{amsart}

\theoremstyle{plain}
\newtheorem{theorem}{Theorem}
\newtheorem{proposition}
{Proposition}
\newtheorem{lemma}
{Lemma}

\theoremstyle{definition}

\theoremstyle{remark}

\usepackage{cite}
\usepackage{amssymb}
\usepackage{amsthm}
\usepackage{amsmath}

\begin{document}

\title
[On positiveness and contractiveness of the integral operator]
{On positiveness and contractiveness of the integral operator 
arising from the beam deflection problem on elastic foundation}

\author{Sung Woo Choi}
\address{Department of Mathematics\\
Duksung Women's University\\
Seoul 132-714, Korea}
\email{swchoi@duksung.ac.kr}


\begin{abstract}
We provide a complete proof
that there are no nontrivial eigenvalues
of the integral operator $\mathcal{K}_l$
outside the interval $(0,1/k)$.
$\mathcal{K}_l$ arises naturally
from the deflection problem of a beam
with length $l$ resting horizontally on an elastic foundation
with spring constant $k$, 
while some vertical load is applied to the beam.
\end{abstract}

\maketitle

\section{Introduction}

We consider the vertical deflection $u(x)$ 
of a linear-shaped beam with length $l > 0$ resting horizontally
on an elastic foundation.
The beam is subject to the downward load distribution
$w(x)$ applied vertically on the beam.
The given elastic foundation follows Hooke's law
with spring constant $k > 0$,
so that $k \cdot u(x)$ is the spring force distribution 
by the elastic foundation.
Let the constants $E$ and $I$ be the Young's modulus 
and the mass moment of inertia of the beam respectively, 
so that $EI$ is the flexural rigidity of the beam.
According to the classical Euler beam theory,
the resulting deflection $u(x)$ is a solution
of the following fourth-order linear ODE:
\begin{equation}
\label{equation_linear}
E I \, \frac{d^4 u(x)}{dx^4} + k \cdot u(x) = w(x).
\end{equation}

The beam deflection problem described above has been 
one of the cornerstones
of mechanical engineering~\cite{Alvesetal2009,Beaufait1980,
Galewski2011,Grossinhoetal2011,Hetenyi1946,Kuoetal1994,Mirandaetal1966,
Timoshenko1926,Timoshenko1955,Ting1982}.
In fact, when the length of the beam is infinite, 
\eqref{equation_linear}
with the boundary condition
$
\lim_{x \to \pm \infty} u(x)
 =
\lim_{x \to \pm \infty} u^\prime(x)
 =
0
$
has the following closed form solution~\cite{Greenberg}:
\[
u(x)
 =
\int_{-\infty}^\infty
K\left( \left| x - \xi \right| \right) w(\xi) \, d\xi.
\]
Here, the kernel function $K(\cdot)$ is
\[
K(y)
 :=
\frac{\alpha}{2k} \exp\left( -\frac{\alpha}{\sqrt{2}} y \right)
\sin\left( \frac{\alpha}{\sqrt{2}} y + \frac{\pi}{4} \right),
\]
where
$\alpha
 :=
\sqrt[4]{k/(EI)}$.
By analyzing the integral operator
$\mathcal{K}$ defined by 
\[
\mathcal{K}[u](x)
 :=
\int_{-\infty}^\infty K\left( |x - \xi| \right) u(\xi) \, d\xi,
\]
Choi et al.~\cite{ChoiJang}
obtained an existence and uniqueness result
for the solution
of the following nonlinear and nonuniform generalization of
\eqref{equation_linear} for infinitely long beam:
\begin{equation}
\label{equation_nonlinear}
E I \, \frac{d^4 u(x)}{dx^4} + \phi\left( u(x),x \right) = w(x).
\end{equation}

To deal with the more practical problem
of the nonlinear and nonuniform beam deflection
with a finite length $l > 0$,
it is important to analyze the integral operator
$\mathcal{K}_l$ defined by
\[
\mathcal{K}_l[u](x)
 :=
\int_{-l}^l K\left( |x - \xi| \right) u(\xi) \, d\xi.
\]
Recently,
Choi~\cite{ChoiI,ChoiII} performed analysis
on the eigenstructure of $\mathcal{K}_l$
as a linear operator on the Hilbert space $L^2[-l,l]$
of the square-integrable complex functions on
$[-l,l]$.

\begin{proposition}[\cite{ChoiII}]
The eigenvalues of $\mathcal{K}_l$ inside the real interval
$(0,1/k)$ are
$\mu_1/k > \nu_1/k > \mu_2/k > \nu_2/k > \cdots \searrow 0$,
and
$\mu_n \sim \nu_n \sim n^{-4}$
as $n \to \infty$.
\end{proposition}

Since the operator $\mathcal{K}_l$ is self-adjoint,
all of its eigenvalues are real.
Note that $0$ is the trivial eigenvalue.
In fact, it is shown in \cite{ChoiI} that
$0$ is the only eigenfunction
corresponding to the trivial eigenvalue $0$,
and $1/k$ is not
an eigenvalue of $\mathcal{K}_l$.
About the eigenvalues of $\mathcal{K}_l$ in
$(-\infty,0) \cup (1/k,\infty)$,
they obtained a characteristic equation
in terms of specific functions
$\psi_L(\kappa)$ and $q(\kappa)$
defined in Section~\ref{section_preliminaries}.

\begin{proposition}[\cite{ChoiI}]
\label{proposition_characteristic}
$\lambda \in (-\infty,0) \cup (1/k,\infty)$ is an
eigenvalue of $\mathcal{K}_l$,
if and only if $\psi_L(\kappa) = q(\kappa)$,
where $\kappa = \sqrt[4]{1 - 1/(\lambda k)} > 0$ 
and $L = 2\sqrt{2}l\alpha$.
\end{proposition}

In this paper,
we provide a complete proof of the fact
\begin{equation}
\label{equation_psi>q}
\psi_L(\kappa) > q(\kappa)
\quad
\text{for every }
\kappa > 0
\text{ and for every }
L > 0,
\end{equation}
from which the following result follows immediately
by Proposition~\ref{proposition_characteristic}.

\begin{theorem}
\label{theorem_main}
There are no nontrivial eigenvalues of the operator $\mathcal{K}_l$
outside the interval $(0,1/k)$.
\end{theorem}

Theorem~\ref{theorem_main} implies that
the operator $\mathcal{K}_l$ is positive
and contractive in dimension-free sense,
which is relevant
to the existence and the uniqueness of
the solution to the nonlinear and nonuniform problem
\eqref{equation_nonlinear}.
We remark that
the proof of Lemma 3.2 in \cite{ChoiI},
which also asserts \eqref{equation_psi>q},
was incomplete in that it
only amounts to showing that
$\psi_L(\kappa) > q(\kappa)$
for every {\em sufficiently small} $\kappa > 0$
for every $L > 0$,
which is indeed far from complete.
However, our proof of
\eqref{equation_psi>q} indicates that
the conclusions of \cite{ChoiI},
including Lemma 3.2 and Theorems 4.1, 4.2 therein,
remain unchanged.

\section{Preliminaries}
\label{section_preliminaries}

For $\kappa \geq 0$, define
\begin{align}
q(\kappa)
 &=
\frac{(\kappa - 1)^2}{(\kappa + 1)^2},
\label{equation_q} \\
\psi_L(\kappa)
 &=
e^{L \kappa} \cdot f\left( \cos{g_L(\kappa)} \right),
\label{equation_psi}
\end{align}
where
\begin{equation}
\label{equation_f}
f(t)
 =
(2 - t) - \sqrt{(2-t)^2 - 1}.
\end{equation}
Here,
$L := 2\sqrt{2}l\alpha$,
$l$, $\alpha$ are {\em positive} constants,
and
the function $g_L$,
parametrized by $L > 0$,
is one-to-one and onto
from $[0,\infty)$ to $[0,\infty)$
with $g_L(0) = 0$.
Specifically,
$g_L$, which was denoted by $g$ in \cite{ChoiI},
is defined as follows:
\begin{equation}
\label{equation_g}
g_L(\kappa)
 =
L \kappa - \hat{g}(\kappa),
\end{equation}
where
\begin{equation}
\label{equation_ghat}
\hat{g}(\kappa)
 =
\begin{cases}
\arctan
\left\{
\frac{4 \kappa \left( \kappa^2 - 1 \right)}
{\kappa^4 - 6 \kappa^2 + 1}
\right\}
&
\text{if } 0 \leq \kappa < \sqrt{2} - 1, \\
-\frac{\pi}{2}
&
\text{if } \kappa = \sqrt{2} - 1, \\
-\pi
+
\arctan
\left\{
\frac{4 \kappa \left( \kappa^2 - 1 \right)}
{\kappa^4 - 6 \kappa^2 + 1}
\right\}
&
\text{if } \sqrt{2} - 1 < \kappa < 
 \sqrt{2} + 1, \\
-\frac{3\pi}{2}
&
\text{if } \kappa = \sqrt{2} + 1, \\
-2\pi
+
\arctan
\left\{
\frac{4 \kappa \left( \kappa^2 - 1 \right)}
{\kappa^4 - 6 \kappa^2 + 1}
\right\}
&
\text{if } \kappa > \sqrt{2} + 1.
\end{cases}
\end{equation}
Here, the branch of $\arctan$ is taken such that
$\arctan(0) = 0$.
As is shown in \cite{ChoiI},
$\hat{g}$ is continuous and differentiable on $[0,\infty)$,
and is strictly decreasing 
from $\hat{g}(0) = 0$ 
to $\lim_{\kappa \to \infty}{\hat{g}(\kappa)} = -2\pi$.
In fact, we have \cite[pp. 43--44]{ChoiI}
\begin{align}
\hat{g}^\prime(\kappa)
 &=
-\frac{4}{\kappa^2 + 1},
\label{equation_ghat'} \\
\left. g_L \!\right.^\prime\!(\kappa)
 &=
L
+
\frac{4}{\kappa^2 + 1}.
\label{equation_g'}
\end{align}
The inverse $g_L^{-1}$ of $g_L$ is 
differentiable,
and is one-to-one and onto
from $[0,\infty)$ to $[0,\infty)$
with
$g_L^{-1}(0) = 0$.

Note that the function $q$ is differentiable.
The function $\psi_L$ is continuous, 
but is only piecewise differentiable.
(See Lemma~\ref{lemma_psi} (a) and its proof below.)
The following observation,
which is immediate
from the intermediate value theorem
and the mean value theorem,
plays a key role in our proof
of \eqref{equation_psi>q},
and hence
Theorem~\ref{theorem_main}.

\begin{proposition}
\label{proposition_key}
Suppose $\xi$ and $\eta$ are continuous and
piecewise differentiable functions
on $[a,b]$
satisfying
$\xi(a) \geq \eta(a)$ and 
$\xi(b) \geq \eta(b)$, 
and possible discontinuities
of $\xi^\prime$ and $\eta^\prime$ are discrete.
Suppose the equation
$\xi(\kappa) \leq \eta(\kappa)$
has a solution in $(a,b)$,
and $\xi$ and $\eta$ are differentiable
at every such solution.
Then there exists $\kappa_0$ in $(a,b)$ such that
$\xi\left( \kappa_0 \right) \leq \eta\left( \kappa_0 \right)$
and
$\xi^\prime\left( \kappa_0 \right)
=
\eta^\prime\left( \kappa_0 \right)$.
\end{proposition}

%
%

\section{The functions $\psi_L$ and $q$}

We first examine
properties of the functions $\psi_L$ and $q$.
From \eqref{equation_q}, we have
\begin{align}
q^\prime(\kappa)
 &=
\left\{ 
\frac{\left( \kappa - 1 \right)^2}{\left( \kappa + 1 \right)^2} 
\right\}^\prime
 =
\frac
{
2\left( \kappa - 1 \right) \cdot \left( \kappa + 1 \right)^2
-
\left( \kappa - 1 \right)^2 \cdot 2 \left( \kappa + 1 \right)
}
{\left( \kappa + 1 \right)^4}
\nonumber \\
 &=
\frac
{
2
\left( \kappa - 1 \right)
\left\{
\left( \kappa + 1 \right)
-
\left( \kappa - 1 \right)
\right\}
}
{\left( \kappa + 1 \right)^3}
 =
\frac{4\left( \kappa - 1 \right)}{\left( \kappa + 1 \right)^3}.
\label{equation_q'} 
\end{align}
The properties of the function $q(\kappa)$
that we need,
are summarized in Lemma~\ref{lemma_q},
whose proof is immediate from \eqref{equation_q}
and \eqref{equation_q'}.

\begin{lemma}
\label{lemma_q}
$q$ is strictly decreasing on $[0,1]$ 
from $q(0) = 1$ to $q(1) = 0$,
and strictly increasing on $[1,\infty)$ approaching $1$.
In particular,
$0 \leq q(\kappa) < 1$ for $\kappa > 0$.
\end{lemma}

Note that the function $f$ in \eqref{equation_f}
is continuous and positive.
It is differentiable except at $t = 1$.
In fact, we have
\begin{align}
f^\prime(t)
 &=
-1
-
\frac{2(2-t) \cdot (-1)}{2\sqrt{(2-t)^2 - 1}}
 =
-1
+
\frac{2-t}{\sqrt{(2-t)^2 - 1}}
 =
\frac{f(t)}{\sqrt{(2-t)^2 - 1}}
\label{equation_f'} \\
 &\geq
0,
\nonumber
\end{align}
and hence
$f$ is increasing.
It follows that
\begin{equation}
\label{equation_fminmax}
0
 <
3 - 2\sqrt{2}
 \leq
f(\cos{g_L(\kappa)})
 \leq
1
\quad
\text{for }
\kappa > 0,
\end{equation}
since $-1 \leq \cos{g_L(\kappa)} \leq 1$
and $f(-1) = 3 - 2\sqrt{2}$, $f(1) = 1$.
So
$\psi_L(\kappa)
 =
e^{L \kappa} f\left( \cos{\kappa} \right)
 \geq
\left( 3 - 2\sqrt{2} \right) e^{L \kappa}$,
and hence we have
\begin{equation}
\label{equation_psi>0}
\psi_L(\kappa)
 >
0
\quad
\text{for }
\kappa > 0,
\
L
 >
0,
\end{equation}
\begin{equation}
\label{equation_psitoinfty}
\lim_{\kappa \to \infty}{\psi_L(\kappa)} = \infty
\quad
\text{ for }
L > 0.
\end{equation}
By \eqref{equation_f'}, we have
\begin{align}
\left. \psi_L \!\right.^\prime\!(\kappa)
 &=
e^{L \kappa}
\left\{
L
\cdot
f\left( \cos{g_L(\kappa)} \right)
+
f^\prime\left( \cos{g_L(\kappa)} \right)
\cdot
\left( -\sin{g_L(\kappa)} \right)
\cdot
\left. g_L \!\right.^\prime\!(\kappa)
\right\}
\nonumber \\
 &=
e^{L \kappa}
\left[
L
\cdot
f\left( \cos{g_L(\kappa)} \right)
+
\frac
{
f\left( \cos{g_L(\kappa)} \right)
\cdot
\left( -\sin{g_L(\kappa)} \right)
\cdot
\left. g_L \!\right.^\prime\!(\kappa)
}
{
\sqrt{\left( 2 - \cos{g_L(\kappa)} \right)^2 - 1}
}
\right]
\nonumber \\
 &=
\psi_L(\kappa)
\left\{
L
-
\frac
{\sin{g_L(\kappa)}}
{\sqrt{\left( 2 - \cos{g_L(\kappa)} \right)^2 - 1}}
\cdot
\left. g_L \!\right.^\prime\!(\kappa)
\right\}.
\label{equation_psi'}
\end{align}
Using the identity
\begin{equation}
\label{equation_3-cos}
\left( 2 - \cos{t} \right)^2 - 1
 =
\cos^2{t} - 4 \cos{t} + 3
 =
\left( 1 - \cos{t} \right)
\left( 3 - \cos{t} \right),
\end{equation}
we have
\begin{align}
\lim_{t \to 0\pm}
{
\frac
{\sin{t}}
{
\sqrt{\left( 2 - \cos{t} \right)^2 - 1}}
}
 &=
\lim_{t \to 0\pm}
{
\frac
{
\pm
\sqrt{
\left( 1 - \cos{t} \right)
\left( 1 + \cos{t} \right)
}
}
{
\sqrt{
\left( 1 - \cos{t} \right)
\left( 3 - \cos{t} \right)
}
}
}
\nonumber \\
 &=
\pm
\lim_{t \to 0\pm}
{
\frac
{
\sqrt{
\left( 1 + \cos{t} \right)
}
}
{
\sqrt{
\left( 3 - \cos{t} \right)
}
}
}
 =
\pm 1.
\label{equation_lim0pm}
\end{align}
Since
\begin{align*}
\lefteqn{
\left(
\frac{\sin{t}}{\sqrt{(2-\cos{t})^2 - 1}}
\right)^\prime
 =
\frac
{
\cos{t} 
\cdot 
\sqrt{(2-\cos{t})^2 - 1}
-
\sin{t}
\cdot
\frac
{2 (2 - \cos{t}) \cdot \sin{t}}
{2 \sqrt{(2-\cos{t})^2 - 1}}
}
{(2-\cos{t})^2 - 1}
} \\
 &\qquad =
\frac
{
\cos{t} 
\cdot 
\left\{ (2-\cos{t})^2 - 1 \right\}
-
\left( 1 - \cos^2{t} \right)
\left( 2 - \cos{t} \right)
}
{
\sqrt{\left( 2 - \cos{t} \right)^2 - 1}^3
} \\
 &\qquad =
\frac
{
-2 \cos^2{t}
+
4 \cos{t}
-
2
}
{
\sqrt{\left( 2 - \cos{t} \right)^2 - 1}^3
}
 =
-
\frac
{
2
\left( 1 - \cos{t} \right)^2
}
{
\sqrt{\left( 2 - \cos{t} \right)^2 - 1}^3
}
 \leq
0,
\end{align*}
the periodic function
$\sin{t}\left/\sqrt{\left( 2 - \cos{t} \right)^2 - 1}\right.$ 
is strictly decreasing
on $(0,2\pi)$,
and hence, together with
\eqref{equation_lim0pm},
we have
\begin{equation}
\label{equation_psi'minmax}
-1
 \leq
\frac
{\sin{t}}
{\sqrt{\left( 2 - \cos{t} \right)^2 - 1}}
 \leq
1.
\end{equation}

\begin{lemma}
\label{lemma_psi}
(a)
$\psi_L$ is differentiable at 
every $\kappa > 0$ such that
$\psi_L(\kappa) \leq q(\kappa)$.

(b)
$
\left. \psi_L \!\right.^\prime\!(\kappa)
 \geq
-\psi_L(\kappa)
\cdot
4/\left( \kappa^2 + 1 \right)$
for every $\kappa > 0$ 
where $\psi_L$ is differentiable.
\end{lemma}

\begin{proof}
Let $\kappa > 0$.
By \eqref{equation_psi'},
$\psi_L$ is differentiable except
at $g_L^{-1}(2\pi n)$ for $n = 1,2,3,\ldots$
For $n = 1,2,3,\ldots$,
$\psi_L\left( g_L^{-1}(2\pi n) \right)
 =
e^{L \cdot g_L^{-1}(2\pi n)}
\cdot
f(2\pi n)
 =
e^{L \cdot g_L^{-1}(2\pi n)}
 >
1$
by \eqref{equation_psi} and \eqref{equation_f},
and
$q\left( g_L^{-1}(2\pi n) \right) < 1$
by Lemma~\ref{lemma_q}.
So
$\psi_L\left( g_L^{-1}(2\pi n) \right)
 > 
q\left( g_L^{-1}(2\pi n) \right)$
for $n = 1,2,3,\ldots$,
which shows (a).

By \eqref{equation_psi'}, \eqref{equation_psi'minmax},
we have
$
\left. \psi_L \!\right.^\prime\!(\kappa)
 \geq
\psi_L(\kappa)
\cdot
\left\{
L
-
\left. g_L \!\right.^\prime\!(\kappa)
\right\}$,
since $\psi_L(\kappa) > 0$ by \eqref{equation_psi>0}
and $\left. g_L \!\right.^\prime\!(\kappa) > 0$
by \eqref{equation_g'}.
Hence (b) follows 
from \eqref{equation_g'}.
\end{proof}

%
%

\section{Proof of the main result}

In proving \eqref{equation_psi>q},
we will divide the cases into the following:
(i) When $0 < \kappa \leq 1$,
and
(ii) when $\kappa > 1$.
The former case is settled with
Lemma~\ref{lemma_kappa<=1} below.

\begin{lemma}
\label{lemma_kappa<=1}
If $0 < \kappa \leq 1$,
then
$\psi_L(\kappa) > q(\kappa)$
for every $L > 0$.
\end{lemma}

\begin{proof}
Note first that 
$\psi_L(1) > 0 = q(1)$
by \eqref{equation_q} and \eqref{equation_psi>0}.
So \eqref{equation_psi>q} holds when $\kappa = 1$.
Note also that $\psi_L(0) = 1 = q(0)$
by \eqref{equation_q} and \eqref{equation_psi}.
Suppose \eqref{equation_psi>q} is not true
for $0 < \kappa < 1$,
so that there exists a solution of
the equation $\psi_L(\kappa) \leq q(\kappa)$
in $(0,1)$
for some $L > 0$.
By Lemma~\ref{lemma_psi} (a),
$\psi_L$ and $q$ are differentiable at 
every such solution.
Thus we can apply Proposition~\ref{proposition_key}
to $\psi_L$ and $q$ on $[0,1]$,
so that there exists $\kappa_0$ in $(0,1)$
satisfying
$\psi_L\left( \kappa_0 \right) \leq q\left( \kappa_0 \right)$,
$\left. \psi_L \!\right.^\prime\!\left( \kappa_0 \right)
 =
q^\prime\left( \kappa_0 \right)$.
So by \eqref{equation_psi>0} and Lemma~\ref{lemma_psi} (b),
we have
\[
q^\prime\left( \kappa_0 \right)
 =
\left. \psi_L \!\right.^\prime\!\left( \kappa_0 \right)
 \geq
-\psi_L\left( \kappa_0 \right)
\cdot
\frac{4}{\kappa_0^2 + 1}
 \geq
-q\left( \kappa_0 \right)
\cdot
\frac{4}{\kappa_0^2 + 1},
\]
and hence
by \eqref{equation_q} and \eqref{equation_q'},
\[
\frac{4\left( \kappa_0 - 1 \right)}{\left( \kappa_0 + 1 \right)^3}
 \geq
-
\frac
{\left( \kappa_0 - 1 \right)^2}
{\left( \kappa_0 + 1 \right)^2}
\cdot
\frac{4}{\kappa_0^2 + 1}.
\]
Since $0 < \kappa_0 < 1$,
this is equivalent to
$
\kappa_0^2 + 1
 \leq
-
\left( \kappa_0^2 - 1 \right)
$,
or
$\kappa_0^2 \leq 0$,
which implies $\kappa_0 = 0$.
This is a contradiction,
and so we conclude $\psi_L(\kappa) > q(\kappa)$
for every $0 < \kappa \leq 1$.
\end{proof}

%
%

For the rest of the paper,
we will deal with the case $\kappa > 1$.
The next result shows the nature of
the equation $\psi_L(\kappa) \leq q(\kappa)$
with respect to $L$.

\begin{lemma}
\label{lemma_nolowerbound}
Suppose 
the equation $\psi_{L_0}(\kappa) \leq q(\kappa)$
has a positive solution for some $L_0 > 0$.
Then,
for each $L$ with $0 < L \leq L_0$,
there exists $\kappa_L > 1$ such that
$\psi_L\left( \kappa_L \right) \leq q\left( \kappa_L \right)$
and
$\left. \psi_L \!\right.^\prime\!\left( \kappa_L \right)
 = q^\prime\left( \kappa_L \right)$.
\end{lemma}

\begin{proof}
Suppose the equation $\psi_{L_0}(\kappa) \leq q(\kappa)$
has a solution $\kappa_0 > 0$ for some $L_0 > 0$.
Note that $\kappa_0 > 1$ by Lemma~\ref{lemma_kappa<=1}.
From \eqref{equation_g},
we have
$
\partial g_L(\kappa)/\partial L
 =
\kappa
$.
So from \eqref{equation_psi} and \eqref{equation_f'},
we have
\begin{align*}
\frac{\partial \psi_L(\kappa)}{\partial L}
 &=
\frac{\partial}{\partial L}
\left\{
e^{L\kappa} \cdot f\left( \cos{g_L(\kappa)} \right)
\right\} \\
 &=
\kappa e^{L \kappa}
\cdot
f\left( \cos{g_L(\kappa)} \right)
+
e^{L \kappa}
\cdot
f^\prime\left( \cos{g_L(\kappa)} \right)
\cdot
\left( -\sin{g_L(\kappa)} \right)
\cdot
\frac{\partial g_L(\kappa)}{\partial L} \\ 
 &=
\kappa 
\cdot 
e^{L \kappa}
f\left( \cos{g_L(\kappa)} \right)
-
e^{L \kappa}
\cdot
\frac
{
f\left( \cos{g_L(\kappa)} \right)
\sin{g_L(\kappa)}
}
{
\sqrt{\left( 2 - \cos{g_L(\kappa)} \right)^2 - 1}
}
\cdot
\kappa \\ 
 &=
\kappa
\cdot
\psi_L(\kappa)
\left\{
1
-
\frac{\sin{g_L(\kappa)}}
{
\sqrt{\left( 2 - \cos{g_L(\kappa)} \right)^2 - 1}
}
\right\}
 \geq
0,
\end{align*}
where we used 
\eqref{equation_psi>0}
and
\eqref{equation_psi'minmax}
for the last inequality.
Thus $\psi_L\left( \kappa_0 \right)$ is increasing 
with respect to $L$,
and hence
$\psi_L\left( \kappa_0 \right)
 \leq 
\psi_{L_0}\left( \kappa_0 \right)
 \leq 
q\left( \kappa_0 \right)
$
for every $L$ such that $0 < L < L_0$.

Note 
that $\psi_L(1) > 0 = q(1)$
for every $L > 0$.
Since $\lim_{\kappa \to \infty}{q(\kappa)} = 1$
by Lemma~\ref{lemma_q}
and $\lim_{\kappa \to \infty}{\psi_L(\kappa)} = \infty$
by \eqref{equation_psitoinfty},
there exists $b_L > x_0 > 1$ such that
$\psi_L\left( b_L \right) > q\left( b_L \right)$
for each $L > 0$.
By Lemma~\ref{lemma_psi} (a),
$\psi_L$ and $q$ are differentiable at 
every $\kappa \in \left( 1, b_L \right)$
such that $\psi_L(\kappa) \leq q(\kappa)$.
Thus, for each $L$ such that $0 < L < L_0$, 
we can apply Proposition~\ref{proposition_key}
to $\psi_L$ and $q$ on $\left[ 1, b_L \right]$,
so that there exists 
$\kappa_L \in \left( 1, b_L \right) \subset (1,\infty)$
satisfying
$\psi_L\left( \kappa_L \right) \leq q\left( \kappa_L \right)$
and
$\left. \psi_L \!\right.^\prime\!\left( \kappa_L \right)
 =
q^\prime\left( \kappa_L \right)$.
\end{proof}

\begin{lemma}
\label{lemma_1+sqrt2}
Suppose 
$\psi_L\left( \kappa \right) 
\leq q\left( \kappa \right)$
for some $\kappa > 0$
and $L > 0$.
Then $\kappa > 1 + \sqrt{2}$.
\end{lemma}

\begin{proof}
For $L > 0$,
the condition 
$\psi_L(\kappa)
 \leq
q(\kappa)$
implies
\[
\frac{\left( \kappa - 1 \right)^2}
{\left( \kappa + 1 \right)^2}
 \geq
e^{L \kappa}
 f\left( \cos{g_L(\kappa)} \right)
 \geq
e^{L \kappa} \left( 3 - 2\sqrt{2} \right)
 >
3 - 2\sqrt{2}
\]
by \eqref{equation_q}, \eqref{equation_psi},
\eqref{equation_fminmax},
and hence
\begin{align*}
0
 &<
\left( \kappa - 1 \right)^2
- 
\left( 3 - 2\sqrt{2} \right) \left( \kappa + 1 \right)^2 \\
 &=
\left( 2\sqrt{2} - 2 \right) \kappa^2
-
2 \left( 4 - 2\sqrt{2} \right) \kappa
+
\left( 2\sqrt{2} - 2 \right) \\
 &=
\left( 2\sqrt{2} - 2 \right)
\left\{
\kappa^2 - 2\sqrt{2} \kappa + 1
\right\} \\
 &=
\left( 2\sqrt{2} - 2 \right)
\left\{ \kappa - \left( \sqrt{2} - 1 \right) \right\}
\left\{ \kappa - \left( \sqrt{2} + 1 \right) \right\}.
\end{align*}
So we have
$\kappa
 <
  \sqrt{2} - 1$
or
$\kappa
 >
  \sqrt{2} + 1$.
It follows that 
$\kappa
 >
  \sqrt{2} + 1$,
since $\kappa > 1$ by Lemma~\ref{lemma_kappa<=1}.
\end{proof}

In view of Lemma~\ref{lemma_nolowerbound}, 
it is legitimate to consider
the behavior of (hypothetical) $\kappa_L$,
as $L \searrow 0$.

\begin{lemma}
\label{lemma_kappatoinfty}
Suppose 
$\psi_L\left( \kappa_L \right) \leq q\left( \kappa_L \right)$
and
$\left. \psi_L \!\right.^\prime\!\left( \kappa_L \right) 
= q^\prime\left( \kappa_L \right)$
with $\kappa_L > 0$.
Then $\lim_{L \to 0+}{\kappa_L} = \infty$.
\end{lemma}

\begin{proof}
Note first that $\kappa_L > 1$ by Lemma~\ref{lemma_kappa<=1}.
From the assumption
$\left. \psi_L \!\right.^\prime\!\left( \kappa_L \right)
 = q^\prime\left( \kappa_L \right)$
and \eqref{equation_psi'}, 
we have
\[
q^\prime\left( \kappa_L \right)
 =
\left. \psi_L \!\right.^\prime\!\left( \kappa_L \right)
 =
\psi_L\left( \kappa_L \right)
\left\{
L
-
\frac
{\sin{g_L\left( \kappa_L \right)}}
{
\sqrt{
\left( 2 - \cos{g_L\left( \kappa_L \right)} \right)^2 - 1}
}
\cdot
\left. g_L \!\right.^\prime\!\left( \kappa_L \right)
\right\}.
\]
Since $q^\prime\left( \kappa_L \right) > 0$ 
by \eqref{equation_q'}
and
$\psi_L\left( \kappa_L \right) > 0$
by \eqref{equation_psi>0},
we have 
\[
L
-
\frac
{\sin{g_L\left( \kappa_L \right)}}
{
\sqrt{
\left( 2 - \cos{g_L\left( \kappa_L \right)} \right)^2 - 1}
}
\cdot
\left. g_L \!\right.^\prime\!\left( \kappa_L \right)
 >
0,
\]
and hence
\[
q^\prime\left( \kappa_L \right)
 \leq
q\left( \kappa_L \right)
\left\{
L
-
\frac
{\sin{g_L\left( \kappa_L \right)}}
{
\sqrt{
\left( 2 - \cos{g_L\left( \kappa_L \right)} \right)^2 - 1}
}
\cdot
\left. g_L \!\right.^\prime\!\left( \kappa_L \right)
\right\}
\]
by the assumption $\psi_L\left( \kappa_L \right)
 \leq q\left( \kappa_L \right)$.
So by \eqref{equation_q}, \eqref{equation_q'}, we have 
\begin{align*}
\frac{4}{\kappa_L^2 - 1}
 &=
\frac{q^\prime\left( \kappa_L \right)}
{q\left( \kappa_L \right)}
 \leq
L
-
\frac
{\sin{g_L\left( \kappa_L \right)}}
{
\sqrt{
\left( 2 - \cos{g_L\left( \kappa_L \right)} \right)^2 - 1}
}
\cdot
\left. g_L \!\right.^\prime\!\left( \kappa_L \right),
\end{align*}
and hence
\begin{equation}
\label{equation_L-}
\left. g_L \!\right.^\prime\!\left( \kappa_L \right) 
\sin{g_L\left( \kappa_L \right)}
 \leq
\left(
L
-
\frac{4}{\kappa_L^2 - 1}
\right)
\sqrt{\left( 2 - \cos{g_L\left( \kappa_L \right)} \right)^2 - 1}.
\end{equation}

If
$
L
-
\frac{4}{\kappa_L^2 - 1}
 \geq
0
$,
which is equivalent to
$
\kappa_L
 \geq
\sqrt{1 + \frac{4}{L}}
$,
then
$\lim_{L \to 0+}{\kappa_L}
 \geq
\lim_{L \to 0+}\sqrt{1 + \frac{4}{L}}
 =
\infty$,
and hence we have
$\lim_{L \to 0+}{\kappa_L}
 =
\infty$.
So we assume
$
L
-
\frac{4}{\kappa^2 - 1}
 <
0$
for the rest of the proof.
Then
the right side,
and hence the left side as well, of \eqref{equation_L-}
becomes negative.
By squaring the both nonnegative sides of
\[
-
\left. g_L \!\right.^\prime\!\left( \kappa_L \right)
\sin{g_L\left( \kappa_L \right)}
 \geq
-
\left(
L
-
\frac{4}{\kappa_L^2 - 1}
\right)
\sqrt{\left( 2 - \cos{g_L\left( \kappa_L \right)} \right)^2 - 1},
\]
we have
\begin{multline*}
\left\{ 
\left. g_L \!\right.^\prime\!\left( \kappa_L \right) 
\right\}^2
\left( 1 - \cos^2{g_L\left( \kappa_L \right)} \right)
 \geq
\left(
L
-
\frac{4}{\kappa_L^2 - 1}
\right)^2
\left\{
\left( 2 - \cos{g_L\left( \kappa_L \right)} \right)^2 - 1
\right\} \\
 =
\left(
L
-
\frac{4}{\kappa_L^2 - 1}
\right)^2
\left\{
\cos^2{g_L\left( \kappa_L \right)} 
- 
4\cos{g_L\left( \kappa_L \right)} + 3
\right\},
\end{multline*}
and hence
\begin{multline*}
0
 \geq
\left\{
\left\{ 
\left. g_L \!\right.^\prime\!\left( \kappa_L \right) 
\right\}^2 
+
\left(
L
-
\frac{4}{\kappa_L^2 - 1}
\right)^2
\right\}
\cos^2{g_L\left( \kappa_L \right)} \\
-
4
\left(
L
-
\frac{4}{\kappa_L^2 - 1}
\right)^2
\cos{g_L\left( \kappa_L \right)}
+
\left\{
3
\left(
L
-
\frac{4}{\kappa_L^2 - 1}
\right)^2
-
\left\{ 
\left. g_L \!\right.^\prime\!\left( \kappa_L \right) 
\right\}^2 
\right\}.
\end{multline*}
So we have 
$\alpha
 \leq
\cos{g_L\left( \kappa_L \right)}
 \leq
\beta$,
where
$\alpha$, $\beta$ are (interchangeably)
\begin{align*}
\lefteqn{
\frac
{
1
}
{
\left\{ 
\left. g_L \!\right.^\prime\!\left( \kappa_L \right) 
\right\}^2 
+
\left(
L
-
\frac{4}{\kappa_L^2 - 1}
\right)^2
}
\left[
2
\left(
L
-
\frac{4}{\kappa_L^2 - 1}
\right)^2
\pm
\left\{
4
\left(
L
-
\frac{4}{\kappa_L^2 - 1}
\right)^4
\right.
\right.
} \\
 &
\left.
\left.
-
\left\{
\left\{ 
\left. g_L \!\right.^\prime\!\left( \kappa_L \right) 
\right\}^2 
+
\left(
L
-
\frac{4}{\kappa_L^2 - 1}
\right)^2
\right\}
\left\{
3
\left(
L
-
\frac{4}{\kappa_L^2 - 1}
\right)^2
-
\left\{ 
\left. g_L \!\right.^\prime\!\left( \kappa_L \right) 
\right\}^2 
\right\}
\right\}^{\frac{1}{2}}
\right]
 \\
 &=
\frac
{
2
\left(
L
-
\frac{4}{\kappa_L^2 - 1}
\right)^2
\pm
\left|
{
\left\{ 
\left. g_L \!\right.^\prime\!\left( \kappa_L \right) 
\right\}^2 
-
\left(
L
-
\frac{4}{\kappa_L^2 - 1}
\right)^2
}
\right|
}
{
\left\{ 
\left. g_L \!\right.^\prime\!\left( \kappa_L \right) 
\right\}^2 
+
\left(
L
-
\frac{4}{\kappa_L^2 - 1}
\right)^2
} \\
 &=
1,
\quad
\frac
{
-
\left\{ 
\left. g_L \!\right.^\prime\!\left( \kappa_L \right) 
\right\}^2 
+
3
\left(
L
-
\frac{4}{\kappa_L^2 - 1}
\right)^2
}
{
\left\{ 
\left. g_L \!\right.^\prime\!\left( \kappa_L \right) 
\right\}^2 
+
\left(
L
-
\frac{4}{\kappa_L^2 - 1}
\right)^2
}.
\end{align*}
Note that 
$\cos{g_L\left( \kappa_L \right)} < 1$
by Lemma~\ref{lemma_psi} (a) and its proof.
Thus we must have
\[
\frac
{
-
\left\{ 
\left. g_L \!\right.^\prime\!\left( \kappa_L \right) 
\right\}^2 
+
3
\left(
L
-
\frac{4}{\kappa_L^2 - 1}
\right)^2
}
{
\left\{ 
\left. g_L \!\right.^\prime\!\left( \kappa_L \right) 
\right\}^2 
+
\left(
L
-
\frac{4}{\kappa_L^2 - 1}
\right)^2
}
 <
1,
\]
which is equivalent to
\[
\left(
L
-
\frac{4}{\kappa_L^2 - 1}
\right)^2
 <
\left\{ 
\left. g_L \!\right.^\prime\!\left( \kappa_L \right) 
\right\}^2
 =
\left(
L
+
\frac{4}{\kappa_L^2 + 1}
\right)^2
\]
by \eqref{equation_g'}.
Since we assumed that
$L - 4/\left( \kappa_L^2 - 1 \right) < 0$,
we have
\[
-\left(
L
-
\frac{4}{\kappa_L^2 - 1}
\right)
 <
L
+
\frac{4}{\kappa_L^2 + 1},
\]
and hence
\[
L
 >
\frac{1}{2}
\left(
\frac{4}{\kappa_L^2 - 1}
-
\frac{4}{\kappa_L^2 + 1}
\right)
 =
\frac{4}{\kappa_L^4 - 1},
\]
which is equivalent to
$
\kappa_L
 >
\sqrt[4]{
1 + \frac{4}{L}
}$.
So
$
\lim_{L \to 0+}{\kappa_L}
 \geq
\lim_{L \to 0+}
\sqrt[4]{
1 + \frac{4}{L}
}
 =
\infty$.
Thus we have
$\lim_{L \to 0+}{\kappa_L} = \infty$,
and the proof is complete.
\end{proof}

\begin{lemma}
\label{lemma_Lkappa}
Suppose 
$\psi_L\left( \kappa_L \right) \leq q\left( \kappa_L \right)$
and
$\left. \psi_L \!\right.^\prime\!\left( \kappa_L \right) 
= q^\prime\left( \kappa_L \right)$
with $\kappa_L > 0$.
Then
$g_L\left( \kappa_L \right) < 2\pi$
and
$\lim_{L \to 0+}{g_L\left( \kappa_L \right)} = 2\pi$.
\end{lemma}

\begin{proof}
From the assumption
$\psi_L\left( \kappa_L \right)
 =
e^{L \kappa_L}
\cdot
f\left( \cos{g_L\left( \kappa_L \right)} \right)
 \leq q\left( \kappa_L \right)$,
we have
\begin{align*}
\frac
{e^{L \kappa_L}}
{q\left( \kappa_L \right)}
 &\leq
\frac
{1}
{f\left( \cos{g_L\left( \kappa_L \right)} \right)}
 =
\frac{1}
{
2 - \cos{g_L\left( \kappa_L \right)} 
- 
\sqrt{\left( 2 - \cos{g_L\left( \kappa_L \right)} \right)^2 - 1}
} \\
 &=
2 - \cos{g_L\left( \kappa_L \right)}
+ 
\sqrt{\left( 2 - \cos{g_L\left( \kappa_L \right)} \right)^2 - 1}.
\end{align*}
Since $\cos{t} = \cos(t - 2\pi) \geq 1 - (t - 2\pi)^2/2$,
we have
$2 - \cos{t}
 \leq 
2 - \left\{ 1 - (t - 2\pi)^2/2 \right\}
 =
1 + (t - 2\pi)^2/2$,
and hence
\begin{align*}
2 - \cos{t} + \sqrt{\left( 2 - \cos{t} \right)^2 - 1}
 &\leq
1 + \frac{(t - 2\pi)^2}{2}
+
\sqrt
{
\left\{ 1 + \frac{(t - 2\pi)^2}{2} \right\}^2 - 1
} \\
 &=
1 + \frac{(t - 2\pi)^2}{2}
+
\sqrt
{
(t - 2\pi)^2
+
\frac{(t - 2\pi)^4}{4}
} \\
 &=
1 + \frac{(t - 2\pi)^2}{2}
+
|t - 2\pi|
\sqrt
{
1
+
\frac{(t - 2\pi)^2}{4}
} \\
 &\leq
1 + \frac{(t - 2\pi)^2}{2}
+
|t - 2\pi|
\left\{
1
+
\frac{(t - 2\pi)^2}{8}
\right\} \\
 &=
1
+
|t - 2\pi|
+
\frac{|t - 2\pi|^2}{2}
+
\frac{|t - 2\pi|^3}{8}
\end{align*}
for every $t \in \mathbb{R}$,
where we used the inequality
$\sqrt{1 + x^2/4} \leq 1 + x^2/8$
for the second inequality.
So we have
\[
\frac
{e^{L \kappa_L}}
{q\left( \kappa_L \right)}
 \leq
1
+
\left| g_L \left( \kappa_L \right) - 2\pi \right|
+
\frac{1}{2}
\left| g_L \left( \kappa_L \right) - 2\pi \right|^2
+
\frac{1}{8}
\left| g_L \left( \kappa_L \right) - 2\pi \right|^3.
\]
Note that, 
since $\kappa_L > 1 + \sqrt{2}$ 
by Lemma~\ref{lemma_1+sqrt2},
\begin{equation}
\label{equation_gLkappaL}
g_L \left( \kappa_L \right) - 2\pi
 = 
L \kappa_L - \hat{g}\left( \kappa_L \right) - 2\pi
 =
L \kappa_L 
- 
\arctan
{
\frac
{4 \kappa_L \left(\kappa_L^2 - 1 \right)}
{\kappa_L^4- 6 \kappa_L^2 + 1}
}
\end{equation}
by \eqref{equation_g}
and \eqref{equation_ghat}.
So from the inequality
$e^x > 1 + x + \frac{x^2}{2} + \frac{x^3}{6}$
for $x > 0$,
we have
\begin{align*}
\lefteqn{
\frac{1}{q\left( \kappa_L \right)}
\left\{
1
+
L \kappa_L
+
\frac{1}{2}
\left( L \kappa_L \right)^2
+
\frac{1}{6}
\left( L \kappa_L \right)^3
\right\}
} \\
 &
\quad <
1
+
\left| 
L \kappa_L
-
\arctan
{
\frac
{4 \kappa_L \left( \kappa_L^2 - 1 \right)}
{\kappa_L^4- 6 \kappa_L^2 + 1}
}
\right|
+
\frac{1}{2}
\left| 
L \kappa_L
-
\arctan
{
\frac
{4 \kappa_L \left( \kappa_L^2 - 1 \right)}
{\kappa_L^4- 6 \kappa_L^2 + 1}
}
\right|^2 \\
 &
\qquad
+
\frac{1}{8}
\left| 
L \kappa_L
-
\arctan
{
\frac
{4 \kappa_L \left( \kappa_L^2 - 1 \right)}
{\kappa_L^4- 6 \kappa_L^2 + 1}
}
\right|^3,
\end{align*}
or equivalently,
\begin{align}
\lefteqn{
24
+
24
L \kappa_L
+
12
\left( L \kappa_L \right)^2
+
4
\left( L \kappa_L \right)^3
}
\nonumber \\
 &
\qquad <
24
q\left( \kappa_L \right)
+
24
q\left( \kappa_L \right)
\left| 
L \kappa_L
-
\arctan
{
\frac
{4 \kappa_L \left( \kappa_L^2 - 1 \right)}
{\kappa_L^4- 6 \kappa_L^2 + 1}
}
\right|
\nonumber \\
 &
\qquad\qquad
+
12
q\left( \kappa_L \right)
\left| 
L \kappa_L
-
\arctan
{
\frac
{4 \kappa_L \left( \kappa_L^2 - 1 \right)}
{\kappa_L^4- 6 \kappa_L^2 + 1}
}
\right|^2
\nonumber \\
 &
\qquad\qquad
+
3
q\left( \kappa_L \right)
\left| 
L \kappa_L
-
\arctan
{
\frac
{4 \kappa_L \left( \kappa_L^2 - 1 \right)}
{\kappa_L^4- 6 \kappa_L^2 + 1}
}
\right|^3.
\label{equation_Lkappa}
\end{align}

Suppose 
\[
L \kappa_L
 \geq
\arctan
{
\frac
{4 \kappa_L \left( \kappa_L^2 - 1 \right)}
{\kappa_L^4- 6 \kappa_L^2 + 1}.
}
\]
Then \eqref{equation_Lkappa} becomes
\begin{align*}
0
 &>
\left\{
4
-
3
q\left( \kappa_L \right)
\right\}
\left( L \kappa_L \right)^3 \\
 &
\quad
+
\left\{
12
+
9
q\left( \kappa_L \right)
\arctan
{
\frac
{4 \kappa_L \left( \kappa_L^2 - 1 \right)}
{\kappa_L^4- 6 \kappa_L^2 + 1}
}
-
12
q\left( \kappa_L \right)
\right\}
\left( L \kappa_L \right)^2 \\
 &
\quad
+
\left\{
24
-
9
q\left( \kappa_L \right)
\arctan^2
{
\frac
{4 \kappa_L \left( \kappa_L^2 - 1 \right)}
{\kappa_L^4- 6 \kappa_L^2 + 1}
}
\right. \\
 &
\quad\qquad\qquad
\left.
+
24
q\left( \kappa_L \right)
\arctan
{
\frac
{4 \kappa_L \left( \kappa_L^2 - 1 \right)}
{\kappa_L^4- 6 \kappa_L^2 + 1}
}
-
24
q\left( \kappa_L \right)
\right\}
L \kappa_L \\
 &
\quad
+
\left\{
24
+
3
q\left( \kappa_L \right)
\arctan^3
{
\frac
{4 \kappa_L \left( \kappa_L^2 - 1 \right)}
{\kappa_L^4- 6 \kappa_L^2 + 1}
}
-
12
q\left( \kappa_L \right)
\arctan^2
{
\frac
{4 \kappa_L \left( \kappa_L^2 - 1 \right)}
{\kappa_L^4- 6 \kappa_L^2 + 1}
}
\right. \\
 &
\quad\qquad\qquad
\left.
+
24
q\left( \kappa_L \right)
\arctan
{
\frac
{4 \kappa_L \left( \kappa_L^2 - 1 \right)}
{\kappa_L^4- 6 \kappa_L^2 + 1}
}
-
24
q\left( \kappa_L \right)
\right\},
\end{align*}
and hence
\begin{equation}
\label{equation_Lkappaabc}
\left( L \kappa_L \right)^3
+
a
\left( L \kappa_L \right)^2
+
b
L \kappa_L
+
c
 <
0,
\end{equation}
where
\begin{align*}
a
 &=
\frac
{12 \left\{ 1 - q\left( \kappa_L \right) \right\}}
{4 - 3 q\left( \kappa_L \right)}
+
\frac
{9 q\left( \kappa_L \right)}
{4 - 3 q\left( \kappa_L \right)}
\arctan
{
\frac
{4 \kappa_L \left( \kappa_L^2 - 1 \right)}
{\kappa_L^4- 6 \kappa_L^2 + 1},
}
 \\
b
 &=
\frac
{24 \left\{ 1 - q\left( \kappa_L \right) \right\}}
{4 - 3 q\left( \kappa_L \right)}
-
\frac
{
q\left( \kappa_L \right)
}
{4 - 3 q\left( \kappa_L \right)}
\arctan
{
\frac
{4 \kappa_L \left( \kappa_L^2 - 1 \right)}
{\kappa_L^4- 6 \kappa_L^2 + 1}
}
\cdot \\
 &
\qquad
\cdot
\left\{
9
\arctan
{
\frac
{4 \kappa_L \left( \kappa_L^2 - 1 \right)}
{\kappa_L^4- 6 \kappa_L^2 + 1}
}
-
24
\right\}, \\
c
 &=
\frac
{24 \left\{ 1 - q\left( \kappa_L \right) \right\}}
{4 - 3 q\left( \kappa_L \right)}
+
\frac
{
3
q\left( \kappa_L \right)
}
{4 - 3 q\left( \kappa_L \right)}
\arctan
{
\frac
{4 \kappa_L \left( \kappa_L^2 - 1 \right)}
{\kappa_L^4- 6 \kappa_L^2 + 1}
}
\cdot \\
 &
\qquad
\cdot
\left\{
\arctan^2
{
\frac
{4 \kappa_L \left( \kappa_L^2 - 1 \right)}
{\kappa_L^4- 6 \kappa_L^2 + 1}
}
-
4
\arctan
{
\frac
{4 \kappa_L \left( \kappa_L^2 - 1 \right)}
{\kappa_L^4- 6 \kappa_L^2 + 1}
}
+
8
\right\}. \\
\end{align*}
Since $\kappa_L > 1 + \sqrt{2}$
and
\begin{align*}
\kappa^4 - 6 \kappa^2 + 1
 &=
\left( \kappa^2 - 1 \right)^2 - 4 \kappa^2
 =
\left( \kappa^2 + 2 \kappa - 1 \right)
\left( \kappa^2 - 2 \kappa - 1 \right) \\
 &=
\left( \kappa + 1 + \sqrt{2} \right)
\left( \kappa + 1 - \sqrt{2} \right)
\left( \kappa - 1 + \sqrt{2} \right)
\left( \kappa - 1 - \sqrt{2} \right),
\end{align*}
we have
$
4 \kappa_L \left( \kappa_L^2 - 1 \right)/
\left( \kappa_L^4- 6 \kappa_L^2 + 1 \right)
 >
0$,
and hence
\[
0
 <
\arctan
{
\frac
{4 \kappa_L \left( \kappa_L^2 - 1 \right)}
{\kappa_L^4- 6 \kappa_L^2 + 1}
}
 <
\frac{\pi}{2}
 \approx
1.5708.
\]
Again since $\kappa_L > 1 + \sqrt{2}$,
we have $0 < q\left( \kappa_L \right) < 1$ by Lemma~\ref{lemma_q},
and hence
\[
\frac
{1 - q\left( \kappa_L \right)}
{4 - 3 q\left( \kappa_L \right)}
 >
0,
\qquad
\frac
{q\left( \kappa_L \right)}
{4 - 3 q\left( \kappa_L \right)}
 >
0.
\]
It follows that $a, b, c > 0$,
which is a contradiction to \eqref{equation_Lkappaabc}
since $L \kappa_L > 0$.
Hence we have
\begin{equation}
\label{equation_Lkappaarctan}
L \kappa_L
 <
\arctan
{
\frac
{4 \kappa_L \left( \kappa_L^2 - 1 \right)}
{\kappa_L^4- 6 \kappa_L^2 + 1}
}.
\end{equation}

By \eqref{equation_gLkappaL} and \eqref{equation_Lkappaarctan},
we have
\[
g_L\left( \kappa_L \right)
 =
L \kappa_L 
-
\arctan
{
\frac
{4 \kappa_L \left( \kappa_L^2 - 1 \right)}
{\kappa_L^4- 6 \kappa_L^2 + 1}
}
+ 
2\pi
 <
2\pi.
\]
Since $L \kappa_L > 0$,
we have
\[
-
\arctan
{
\frac
{4 \kappa_L \left( \kappa_L^2 - 1 \right)}
{\kappa_L^4- 6 \kappa_L^2 + 1}
}
 <
L \kappa_L
-
\arctan
{
\frac
{4 \kappa_L \left( \kappa_L^2 - 1 \right)}
{\kappa_L^4- 6 \kappa_L^2 + 1}
}
 <
0
\]
by \eqref{equation_Lkappaarctan}.
So by Lemma~\ref{lemma_kappatoinfty},
\begin{align*}
0
 &\geq
\lim_{L \to 0+}
\left\{
L \kappa_L
-
\arctan
{
\frac
{4 \kappa_L \left( \kappa_L^2 - 1 \right)}
{\kappa_L^4- 6 \kappa_L^2 + 1}
}
\right\}
 \geq
-
\lim_{L \to 0+}
\arctan
{
\frac
{4 \kappa_L \left( \kappa_L^2 - 1 \right)}
{\kappa_L^4- 6 \kappa_L^2 + 1}
} \\
 &=
-
\lim_{\kappa_L \to \infty}
\arctan
{
\frac
{4 \kappa_L \left( \kappa_L^2 - 1 \right)}
{\kappa_L^4- 6 \kappa_L^2 + 1}
}
 =
0,
\end{align*}
and hence we have
\[
\lim_{L \to 0+}
\left\{
L \kappa_L
-
\arctan
{
\frac
{4 \kappa_L \left( \kappa_L^2 - 1 \right)}
{\kappa_L^4- 6 \kappa_L^2 + 1}
}
\right\}
 =
0.
\]
Thus by \eqref{equation_gLkappaL} again,
we have
\[
\lim_{L \to 0+}{g_L\left( \kappa_L \right)}
 =
\lim_{L \to 0+}
\left\{
L \kappa_L
-
\lim_{L \to 0+}
\arctan
{
\frac
{4 \kappa_L \left( \kappa_L^2 - 1 \right)}
{\kappa_L^4- 6 \kappa_L^2 + 1}
}
\right\}
+
2\pi
 =
2\pi,
\]
which completes the proof.
\end{proof}

Lemma~\ref{lemma_Lkappa}
indicates that
it is enough to consider the case
when $g_L(\kappa) < 2\pi$
to prove \eqref{equation_psi>q}.
We will do the change of the variables
from $\kappa$ to $t$ via
$t = g_L(\kappa)$ for $\kappa \geq 0$,
or equivalently,
$\kappa = g_L^{-1}(t)$ for $t \geq 0$.

\begin{lemma}
\label{lemma_g-1}
Suppose $0 < t < 2\pi$.
Then $\lim_{L \to 0+}{g_L^{-1}(t)} = \hat{g}^{-1}(-t)$,
and
$g_L^{-1}(t) < \hat{g}^{-1}(-t)$ for every $L > 0$.
\end{lemma}

\begin{proof}
From the definition \eqref{equation_g} of $g_L$,
we have
\begin{equation}
\label{equation_g-1}
L \cdot g_L^{-1}(t) - \hat{g}\left( g_L^{-1}(t) \right) = t,
\end{equation}
Differentiating with respect to $L$,
we have
\[
1 \cdot g_L^{-1}(t) 
+ 
L \cdot \frac{\partial}{\partial L} g_L^{-1}(t)
-
\hat{g}^\prime\left( g_L^{-1}(t) \right)
\cdot
\frac{\partial}{\partial L} g_L^{-1}(t)
 =
0,
\]
and hence by \eqref{equation_g} and \eqref{equation_g'},
\[
\frac{\partial}{\partial L} g_L^{-1}(t)
 =
-
\frac
{g_L^{-1}(t)}
{L - \hat{g}^\prime\left( g_L^{-1}(t) \right)}
 =
-
\frac{\kappa}{L - \hat{g}^\prime(\kappa)}
 =
-
\frac{\kappa}{\left. g_L \!\right.^\prime\!(\kappa)}
 <
0,
\]
where we put
$\kappa = g_L^{-1}(t)$.
This shows that
$g_L^{-1}(t)$ is strictly decreasing with respect to $L$
for any fixed $t$,
and consequently, $g_L^{-1}(t)$ is strictly increasing
as $L \searrow 0$.

Suppose $0 < t < 2\pi$.
If $\lim_{L \to 0+}{g_L^{-1}(t)} = \infty$,
then by \eqref{equation_ghat} and \eqref{equation_g-1}, we have
\begin{align*}
2\pi
 >
t
 &=
\lim_{L \to 0+}\left\{ L \cdot g_L^{-1}(t) \right\}
-
\lim_{L \to 0+}
\left\{
\hat{g}\left( g_L^{-1}(t) \right)
\right\} \\
 &=
\lim_{L \to 0+}\left\{ L \cdot g_L^{-1}(t) \right\}
-
\lim_{\kappa \to \infty}
\left\{
\hat{g}\left( \kappa \right)
\right\} \\
 &=
\lim_{L \to 0+}\left\{ L \cdot g_L^{-1}(t) \right\}
 -
(-2\pi)
 \geq
2\pi,
\end{align*}
which is a contradiction.
So $\lim_{L \to 0+}{g_L^{-1}(t)} < \infty$. 
Note from \eqref{equation_g-1} again that
\[
t
 =
\lim_{L \to 0+}{L} 
\cdot 
\lim_{L \to 0+}{g_L^{-1}(t)}
-
\lim_{L \to 0+}
\left\{
\hat{g}\left( g_L^{-1}(t) \right)
\right\}
 =
0
-
\hat{g}
\left(
\lim_{L \to 0+}{g_L^{-1}(t)}
\right),
\]
from which it follows that
$
\lim_{L \to 0+}{g_L^{-1}(t)}
 =
\hat{g}^{-1}(-t)
$.
Since $g_L^{-1}(t)$ is strictly decreasing with respect to $L$,
we have 
$g_L^{-1}(t) < \hat{g}^{-1}(-t)$ for every $L > 0$.
\end{proof}

We remark that, in fact, 
$\lim_{L \to 0+}{g_L^{-1}(t)} = \infty$
for every $t \geq 2\pi$,
whose proof we omit.
For $t \geq 0$,
define
\[
\tilde{\psi}_L(t)
 =
\psi_L\left( g_L^{-1}(t) \right),
\qquad
\tilde{q}_L(t)
 =
q\left( g_L^{-1}(t) \right).
\]
The functions $\tilde{\psi}_L$ and
$\tilde{q}_L$ can be considered as
``mollified'' versions of $\psi_L$ and $q$
as $L \searrow 0$.
From the definitions of $\psi_L$ and $\tilde{\psi}_L$,
we have
\begin{equation}
\label{equation_psitilde>}
\tilde{\psi}_L(t)
 =
e^{L \cdot g_L^{-1}(t)}
f\left( \cos{t} \right)
 >
f\left( \cos{t} \right)
\quad
\text{ for }
t > 0.
\end{equation}

Note that
$\hat{g}^{-1}\left( -3\pi/2 \right) = 1 + \sqrt{2}$
by \eqref{equation_ghat},
and
$g_L^{-1}\left( 3\pi/2 \right)$ is strictly increasing
to $\hat{g}^{-1}\left( -3\pi/2 \right)
= 1 + \sqrt{2}$
as $L$ goes down to $0$ by Lemma~\ref{lemma_g-1}.
It follows that,
for every sufficiently small $L > 0$,
we have
$g_L^{-1}(t) > 1$ for $3\pi/2 < t < 2\pi$.
Since
$q$ is strictly increasing on $(1,\infty)$ by Lemma~\ref{lemma_q},
we have
\begin{equation}
\label{equation_qtilde<}
\begin{split}
\tilde{q}_L(t)
 =
q\left( g_L^{-1}(t) \right)
 <
q\left( \hat{g}^{-1}(-t) \right)
 &\quad\text{ for }
3\pi/2 < t < 2\pi \\
 &
\text{ for every sufficiently small }
L > 0
\end{split}
\end{equation}
by Lemma~\ref{lemma_g-1}.

\begin{lemma}
\label{lemma_tildepsi>tildeq}
For every sufficiently small $L > 0$,
$\tilde{\psi}_L(t)
 >
\tilde{q}_L(t)$
for $3\pi/2 < t < 2\pi$.
\end{lemma}

\begin{proof}
By \eqref{equation_psitilde>} and \eqref{equation_qtilde<},
it is enough to show that
$f\left( \cos{t} \right)
 >
q\left( \hat{g}^{-1}(-t) \right)$
for $3\pi/2 < t < 2\pi$.
Suppose $3\pi/2 < t < 2\pi$.
Note that $\kappa := \hat{g}^{-1}(-t) > 1 + \sqrt{2}$ 
by \eqref{equation_ghat}.
So by \eqref{equation_ghat} again,
we have
\[
-t
 =
\hat{g}(\kappa)
 =
-2\pi
+
\arctan
{
\frac
{4\kappa \left( \kappa^2 - 1 \right)}
{\kappa^4 - 6\kappa^2 + 1}
},
\]
and hence
\begin{equation}
\label{equation_tkappa}
\frac
{4\kappa \left( \kappa^2 - 1 \right)}
{\kappa^4 - 6\kappa^2 + 1}
 =
\tan\left( 2\pi - t \right)
 =
-\tan{t}.
\end{equation}
Note that,
for each
$t \in \left( 3\pi/2, 2\pi \right)$,
we have
$-\tan{t} > 0$,
and
$\kappa$ is the unique positive solution
of \eqref{equation_tkappa} such that $\kappa > 1 + \sqrt{2}$.
Transform \eqref{equation_tkappa} to
\[
-\tan{t}
\cdot
\left( \kappa^4 - 6\kappa^2 + 1 \right)
 =
4\kappa \left( \kappa^2 - 1 \right),
\]
and then to
\[
4 \left( \kappa - \frac{1}{\kappa} \right)
 =
-\tan{t}
\cdot
\left(
\kappa^2 - 6 + \frac{1}{\kappa^2}
\right)
 =
-\tan{t}
\cdot
\left\{
\left( \kappa - \frac{1}{\kappa} \right)^2
-
4
\right\}.
\]
Putting 
\begin{equation}
\label{equation_xkappa}
x = \kappa - \frac{1}{\kappa},
\end{equation}
we have
$4x = -\tan{t} \cdot \left( x^2 - 4 \right)$,
and hence
$\tan{t} \cdot x^2 + 4x - 4\tan{t} = 0$,
which gives
\[
x
 =
\frac{-2 \pm \sqrt{4 + 4\tan^2{t}}}{\tan{t}}
 =
\frac{-2 \cos{t} \pm 2}{\sin{t}}.
\]
Note that $\sin{t} < 0$ for $3\pi/2 < t < 2\pi$.
Since $\kappa > 1$,
we have $x > 0$ by \eqref{equation_xkappa},
and hence
\begin{equation}
\label{equation_x}
x
 =
\frac{-2 \cos{t} - 2}{\sin{t}}
 =
\frac{-2 \left( 1 + \cos{t} \right)}{\sin{t}}.
\end{equation}
Substituting \eqref{equation_x} into \eqref{equation_xkappa} again,
we have
\begin{equation}
\label{equation_kappa}
\sin{t} \cdot \kappa^2 
+
2 \left( 1 + \cos{t} \right) \kappa 
- 
\sin{t}
 =
0.
\end{equation}
Solving \eqref{equation_kappa} for $\kappa$,
we have
\begin{align*}
\kappa
 &=
\frac
{
- 
\left( 1 + \cos{t} \right)
\pm
\sqrt
{
\left( 1 + \cos{t} \right)^2
+
\sin^2{t}
}
}
{\sin{t}}
 =
\frac
{
- 
\left( 1 + \cos{t} \right)
\pm
\sqrt{2}
\sqrt
{
1 + \cos{t}
}
}
{\sin{t}}
\end{align*}
Since $\kappa > 0$ and $\sin{t} < 0$,
we finally have
\[
\hat{g}^{-1}(-t)
 =
\kappa
 =
\frac
{
- 
\left( 1 + \cos{t} \right)
-
\sqrt{2}
\sqrt
{
1 + \cos{t}
}
}
{\sin{t}}
 =
\frac
{
\sqrt{1 + \cos{t}}
+
\sqrt{2}
}
{
\sqrt{1 - \cos{t}}
},
\]
and thus by \eqref{equation_q},
\begin{align*}
\lefteqn{
q\left( \hat{g}^{-1}(-t) \right)
} \\
 &=
\left\{
\frac
{
\frac
{
\sqrt{1 + \cos{t}}
+
\sqrt{2}
}
{
\sqrt{1 - \cos{t}}
}
-
1
}
{
\frac
{
\sqrt{1 + \cos{t}}
+
\sqrt{2}
}
{
\sqrt{1 - \cos{t}}
}
+1
}
\right\}^2
 =
\left\{
\frac
{
\sqrt{1 + \cos{t}}
+
\sqrt{2}
-
\sqrt{1 - \cos{t}}
}
{
\sqrt{1 + \cos{t}}
+
\sqrt{2}
+
\sqrt{1 - \cos{t}}
}
\right\}^2 \\
 &=
\left\{
\frac
{
\sqrt{1 + \cos{t}}
+
\sqrt{2}
-
\sqrt{1 - \cos{t}}
}
{
\sqrt{1 + \cos{t}}
+
\sqrt{2}
+
\sqrt{1 - \cos{t}}
}
\cdot
\frac
{
\sqrt{1 + \cos{t}}
+
\sqrt{2}
-
\sqrt{1 - \cos{t}}
}
{
\sqrt{1 + \cos{t}}
+
\sqrt{2}
-
\sqrt{1 - \cos{t}}
}
\right\}^2 \\
 &=
\frac
{1}
{
\left\{
\left( 1 + \cos{t} \right)
+
2\sqrt{2}
\sqrt{1 + \cos{t}}
+
2
-
\left( 1 - \cos{t} \right)
\right\}^2
}
\cdot \\
 &
\qquad
\cdot
\left\{
\left( 1 + \cos{t} \right)
+
\left( 1 - \cos{t} \right)
+
2
+
2\sqrt{2}
\sqrt{1 + \cos{t}}
\right. \\
 &
\qquad\qquad
\left.
-
2\sqrt{2}
\sqrt{1 - \cos{t}}
-
2
\sqrt{1 - \cos{t}}
\sqrt{1 + \cos{t}}
\right\}^2 \\
 &=
\left\{
\frac
{
2\sqrt{2}
\left( \sqrt{1 + \cos{t}} + \sqrt{2} \right)
-
2
\sqrt{1 - \cos{t}}
\left( \sqrt{1 + \cos{t}} + \sqrt{2} \right)
}
{
2\sqrt{1 + \cos{t}}
\left( \sqrt{1 + \cos{t}} + \sqrt{2} \right)
}
\right\}^2 \\
 &=
\left\{
\frac
{
\sqrt{2}
-
\sqrt{1 - \cos{t}}
}
{
\sqrt{1 + \cos{t}}
}
\right\}^2
 =
\frac
{
3 - \cos{t} - 2\sqrt{2} \sqrt{1 - \cos{t}}
}
{1 + \cos{t}}.
\end{align*}
By \eqref{equation_f},
it remains to show that
\[
2 - \cos{t} - \sqrt{\left( 2 - \cos{t} \right)^2 - 1}
 >
\frac
{
3 - \cos{t} - 2\sqrt{2} \sqrt{1 - \cos{t}}
}
{1 + \cos{t}}
\]
for $3\pi/2 < t < 2\pi$,
which is done by
the following series of equivalent transformations:
\[
-\cos^2{t} + \cos{t} + 2
-
\left( 1 + \cos{t} \right)
\sqrt{\left( 2 - \cos{t} \right)^2 - 1}
 >
3 - \cos{t} - 2\sqrt{2} \sqrt{1 - \cos{t}},
\]
\[
\left( 1 - \cos{t} \right)^2
+
\left( 1 + \cos{t} \right)
\sqrt{\left( 1 - \cos{t} \right) \left( 3 - \cos{t} \right)}
 <
2\sqrt{2} \sqrt{1 - \cos{t}},
\]
\[
\sqrt{1 - \cos{t}}^3
+
\left( 1 + \cos{t} \right)
\sqrt{3 - \cos{t}}
 <
2\sqrt{2},
\]
\[
\left( 1 - \cos{t} \right)^3
 <
8
+
\left( 1 + \cos{t} \right)^2
\left( 3 - \cos{t} \right)
-
4\sqrt{2}
\left( 1 + \cos{t} \right)
\sqrt{3 - \cos{t}},
\]
\[
2\cos^2{t} - 8\cos{t} - 10
 <
-4\sqrt{2}
\left( 1 + \cos{t} \right)
\sqrt{3 - \cos{t}},
\]
\[
\left( 1 + \cos{t} \right)
\left( 5 - \cos{t} \right)
 >
2\sqrt{2}
\left( 1 + \cos{t} \right)
\sqrt{3 - \cos{t}},
\]
\[
\cos^2{t} - 10\cos{t} + 25
 >
8
\left( 3 - \cos{t} \right),
\]
\[
\cos^2{t} - 2\cos{t} + 1
 >
0,
\]
where we used \eqref{equation_3-cos} for the second inequality.
\end{proof}

We now have all the ingredients needed to prove
\eqref{equation_psi>q},
which implies
Theorem~\ref{theorem_main}.

\begin{proof}[Proof of Theorem~\ref{theorem_main}]
By Proposition~\ref{proposition_characteristic},
it is sufficient to show \eqref{equation_psi>q}.
Suppose \eqref{equation_psi>q} is false,
so that
the equation $\psi_{L_0}(\kappa) \leq q(\kappa)$
has a positive solution for some $L_0 > 0$.
Then by Lemma~\ref{lemma_nolowerbound},
there exists $\kappa_L$
satisfying
$\psi_L\left( \kappa_L \right) \leq q\left( \kappa_L \right)$
and
$\left. \psi_L \!\right.^\prime\!\left( \kappa_L \right)
 = q^\prime\left( \kappa_L \right)$
for $0 < L < L_0$.
Let $t_N := g_L\left( \kappa_L \right)$
for $0 < L < L_0$.
By Lemma~\ref{lemma_Lkappa},
we have $3\pi/2 < t_L < 2\pi$
for every sufficiently small $L > 0$.
So by Lemma~\ref{lemma_tildepsi>tildeq},
we have
$\tilde{\psi}_L\left( t_L \right)
 >
\tilde{q}_L\left( t_L \right)$,
and hence
\[
\psi_L\left( \kappa_L \right)
 =
\psi_L\left( g_L^{-1}\left( t_L \right) \right)
 =
\tilde{\psi}_L\left( t_L \right)
 >
\tilde{q}_L\left( t_L \right)
 =
q\left( g_L^{-1}\left( t_L \right) \right)
 =
q\left( \kappa_L \right)
\]
for every sufficiently small $L > 0$.
This is a contradiction to the result that
$\psi_L\left( \kappa_L \right)
 \leq
q\left( \kappa_L \right)$
for $0 < L < L_0$.
Thus we conclude that \eqref{equation_psi>q} is true.
\end{proof}

\end{document}